\documentclass[%
 reprint,
 amsmath,amssymb,
 aps,
]{revtex4-2}

\usepackage{graphicx}
\usepackage{dcolumn}
\usepackage{bm}
\usepackage{mathtools}
\usepackage{amsthm}
\usepackage{amsmath}
\usepackage{bbm}

\newtheorem{theorem}{Theorem}
\newtheorem{lemma}{Lemma}
\newtheorem{proposition}{Proposition}

\newcommand{\naturals}{\mathbb{N}}
\newcommand{\integers}{\mathbb{Z}}
\newcommand{\probability}{\mathbb{P}}
\newcommand{\expectation}{\mathbb{E}}

\newcommand{\sgn}{\text{sgn}}
\newcommand{\geom}{\text{Geom}}
\newcommand{\Mod}[1]{\ (\mathrm{mod}\ #1)}

\newcommand{\Uc}{\mathcal{U}}
\newcommand{\Rc}{\mathcal{R}}
\newcommand{\Tc}{\mathcal{T}}
\newcommand{\Ec}{\mathcal{E}}

\newcommand{\Wc}{\mathcal{W}}
\newcommand{\Sc}{\mathcal{S}}

\begin{document}

\preprint{APS/123-QED}

\title{First Passage with Restart in Discrete Time:\\with applications to biased random walks on the half-line}

\author{Jason M. Flynn}
 \email{jasonmflynn@ufl.edu}
\author{Sergei S. Pilyugin}%
 \email{pilyugin@ufl.edu}
\affiliation{%
 University of Florida,\\
 Department of Mathematics
}%

\date{\today}

\begin{abstract}
In recent years, it has been well-established that adding a restart mechanism can alter the first passage statistics of a stochastic processes in useful and interesting ways. Though different mechanisms have been investigated, we derive a probability generating function for a discrete-time First Passage process Under Restart and use it to examine two examples, including a biased random walk on the non-negative integers.
\end{abstract}

\maketitle


\section{Introduction}\label{sec:introduction}

Consider the search for a pair of eyeglasses. Perhaps we look in the two or three most likely places without success, and resign ourselves to checking the top shelf of the refrigerator. However, particularly if the missing eyeglasses are ours, we know that the target may well have been missed in any of the previous locations. We might find our target more quickly if, from time to time, we go back and start the search again.

While the mental framework of using restart to shorten the mean time of search is a useful analogy, this principle can be used to alter the dynamics of many kinds of First Passage (FP) processes across physics \cite{gupta2014fluctuating}, chemistry \cite{reuveni2014role}, biology \cite{roldan2016stochastic} and computer science \cite{huang2007effect}.

To solidify this concept, suppose we have some stochastic process $\Wc_n$ with state space $\Sc$ that starts in some initial set $A\subset\Sc$, by which we mean $\Wc_0\in A$. Then suppose the FP characteristics of $\Wc_n$ into some target set $B\subset\Sc$ are known with the hitting time of this underlying process denoted by $\Uc$. Our main interest becomes whether restarting this process at random intervals might change the mean FP time on some external clock. That is, we define a process $\Wc_n^*$ that starts at $A$ and returns to $A$ at randomly determined times we call restarts, and which has the same dynamics as $\Wc_n$ between those restart events.

For this paper, we assume that the process evolves in discrete time, meaning $n\in\naturals\coloneqq\{0,1,2,\ldots\}$, and that $A\cap B=\emptyset$. For the first assumption, we note that other authors have already spent much time investigating the continuous time case \cite{evans2020stochastic}. For the second, it is helpful for some our results to assume that the FP time of the underlying process is almost surely not 0 ($\probability(\Uc=0)=0)$. Up through Section \ref{sec:hitting_times}, it would be relatively simple to relax this condition, but it is critical for part of Section \ref{sec:ET<EU} (as we note there).

\section{Defining the FPUR} \label{sec:definitions}
In this paper, we largely work within the frame established by Pal and Reuveni in \cite{pal2017first}, and parallel the recent work from Bonomo and Pal in \cite{bonomo2021first} by giving a recursive definition for the hitting time of the First Passage Under Restart (FPUR) process and then deriving its generating function. We denote by $\Uc$ the FP time of the underlying process, by $\Rc$ the time until the next restart, and by $\Tc$ the FP time of the process with restart. Using $\Tc^*$ to represent an independent and identically distributed copy of $\Tc$, we have

\begin{equation} \label{eq:recursive_definition}
    \Tc = 
    \begin{cases}
    \Uc & \text{if }  \Rc > \Uc \\
    \Rc + \Tc^* & \text{if }  \Rc \le \Uc
    \end{cases}.
\end{equation}
In other words, if the underlying process reaches the target set before the restart occurs, then the FPUR also concludes at that time. In the case, however, that the restart occurs before or simultaneously to the underlying process finishing, the time of the restart is noted, and $\Uc$ and $\Rc$ are drawn anew from their respective distributions.  It's worth emphasizing that, in the case that the first passage of the underlying process and the restart occur at the same time, the restart ``wins" the tie, and the process is reset to its initial position. This is an arbitrary decision with some consequences that will be discussed when they come up. In \cite{bonomo2021first}, Bonomo and Pal have a similar derivation in which the weak and sharp inequalities are reversed in the recursive definition.

Even before deriving the generating function for $\Tc$, much can be seen directly from (\ref{eq:recursive_definition}), including a simple expression for $\expectation[\Tc]$, with $p_r \coloneqq \probability(\Rc\le\Uc)$.
\begin{align*}
    \expectation[\Tc] & = \expectation[\Uc \mid \Rc>\Uc]\left(1-p_r\right) + \expectation[\Rc + \Tc^* \mid \Rc \le \Uc]p_r \\
    & = \frac{1}{1-p_r}\left( \expectation[\Uc \mid \Rc>\Uc]\left(1-p_r\right) + \expectation[\Rc \mid \Rc \le \Uc]p_r \right)
\end{align*}
This admits two quick and useful interpretations:
\begin{align}
    \expectation[\Tc] & = \frac{\expectation[\Uc \wedge \Rc]}{1 - p_r},\text{ and} \label{eq:min_formula)}\\
    \expectation[\Tc] & = \expectation[\Uc \mid \Rc > \Uc] + \frac{p_r}{1-p_r}\expectation[\Rc \mid \Rc \le \Uc]\label{eq:num_restart_formula}.
\end{align}
The first is very concise, and the second gives a clearer picture of the restart's effect, where $\frac{p_r}{1-p_r}$ is the expected number of restarts before the FPUR process reaches the target set. We will return to these formulas in later sections.

Before progressing, we must also offer an important definition. We call a restart \underline{preemptive} if $\Rc \le \Uc$ almost surely, or $p_r = 1$. This is a property not just of the restart distribution (in general), but of the interplay between the restart and underlying FP time distributions. In particular, we mostly seek to consider FPUR processes with \underline{non-preemptive} restart, since the alternative is often uninteresting. Since our choice in (\ref{eq:recursive_definition}) means that $\Rc\le\Uc$ prevents the process from terminating, any process with preemptive restart will almost surely never terminate and thus have an infinite mean hitting time. One particularly pathological case is when the underlying process cannot finish in finite time. If $\Uc=\infty$ almost surely, then we say that any restart distribution will be preemptive. To avoid this issue, we assume that $\probability(\Uc=\infty)<1$ in all that follows, unless otherwise noted.

\section{Obtaining the Generating Function for $\Tc$}\label{sec:PGF_for_T}
When characterizing a discrete random variable, $X$, the usefulness of its Probability Generating Function (PGF) can hardly be overstated. Denoting the probability mass function of $X$ by $x(n)$ allows us to define its PGF as follows.
\begin{equation*}
    \tilde x(z) = \sum_{n\ge0} x(n)z^n
\end{equation*}
This expression as a $z$-transform of $x(n)$ allows us to employ numerous techniques for power series with positive coefficients. Additionally, one can easily determine the $k$-th factorial moment of $X$ (that is, $\expectation[X(X-1)\ldots(X-k+1)])$ by taking the corresponding derivative and evaluating the result at $z=1$. Perhaps most usefully, we can evaluate $\tilde x(z)$ and its first derivative to obtain:
\begin{itemize}
    \item $\tilde x(1) = \sum_{n\ge0}x(n) = \probability(X<\infty)$, and
    \item $\tilde x'(1) = \sum_{n\ge0}nx(n) = \expectation[X]$, so long as $\tilde x(1)=1$.
\end{itemize}
It's worth noting that we might generally expect that $\tilde x(1) = 1$, or that the sum of the probability mass equals 1. It is not, however, necessary that this is the case. In particular, if there is some nonzero probability that $X$ does not occur in finite time (such as when our underlying process might escape to infinity), then $\tilde x(1)$ will be the complement of that probability, denoted $\Ec_X \coloneqq \probability(X\text{ is finite})$, often called the hitting probability.

When we address the topic of hitting time for a stochastic process, the probability generating function gives us a powerful tool for characterizing the first passage characteristics. Thus, the ability easily to write the PGF for $\Tc$, denoted $\tilde t(z)$, using only the PGF of $\Uc$ and $\Rc$ would be very valuable. The following lemma provides a formula for $\tilde t(z)$ given $\tilde u(z)$ and $\tilde r(z)$.

\begin{lemma}\label{lemma-pdf_FPUR_gf}
Given the PGF for $\Uc$ and $\Rc$, we can write the probability generating function for $\Tc$ as
\begin{equation*}
    \tilde t(z) = \frac{\tilde u(z) - \sum_{n=0}^\infty z^nu(n)\sum_{i=0}^n r(i)}{1 - \tilde r(z)(1-\Ec_\Uc) - \sum_{n=0}^\infty u(n)\sum_{i=0}^n z^ir(i)}.
\end{equation*}
\end{lemma}
\begin{proof}
Directly from (\ref{eq:recursive_definition}), we can write 
\begin{align*}
    \probability(\Tc=n) & = \probability(\Tc=n \mid \Rc>\Uc)\cdot\probability(\Rc>\Uc) \\
    & \quad + \probability(\Tc = n \mid \Rc\le\Uc)\cdot\probability(\Rc\le\Uc) \\
    & = \probability(\Uc=n \text{ and } \Uc<\Rc) \\
    & \quad + \probability(\Rc+\Tc^*=n \text{ and } \Rc\le \Uc).
\end{align*}

Recalling that $\Uc$, $\Rc$ and $\Tc^*$ are independent allows us to simplify to
\begin{align*}
    t(n) & \coloneqq \probability(\Tc=n) \\
    & = u(n)\left(1-\sum_{i=0}^n r(i)\right) \\
    & \quad + \sum_{i=0}^n \left(r(i)\left[1 - \sum_{j=0}^{i-1} u(j)\right]\right)t^*(n-i).
\end{align*}

We continue with a straightforward method for producing the generating function: taking the $z$-transform. Multiplying both sides of the preceding equation by $z^n$ and summing over $n\in\naturals$ gives 
\begin{align*}
    \sum_{n=0}^\infty z^n t(n) & = \sum_{n=0}^\infty z^nu(n)\left[1 - \sum_{i=0}^n r(i)\right] \\
    & + \sum_{n=0}^\infty z^n\sum_{i=0}^n \left(r(i)\left[1 - \sum_{j=0}^{i-1} u(j)\right]\right)t^*(n-i).
\end{align*}
On the left, we have $\tilde t(z)$ by definition. On the right-hand side, we can focus our attention on the second summand. We split the powers of $z$, change the order of the first two sums to obtain
\begin{align*}
    & \sum_{n=0}^\infty z^n\sum_{i=0}^n \left(r(i)\left[1 - \sum_{j=0}^{i-1} u(j)\right]\right)t^*(n-i) \\
    & = \sum_{n=0}^\infty \sum_{i=0}^n \left(z^ir(i)\left[1 - \sum_{j=0}^{i-1} u(j)\right]\right)z^{n-i}t^*(n-i) \\
    & = \sum_{i=0}^\infty \left(z^ir(i)\left[1 - \sum_{j=0}^{i-1} u(j)\right]\right) \sum_{n=i}^\infty z^{n-i}t^*(n-i) \\
    & = \sum_{i=0}^\infty \left(z^ir(i)\left[1 - \sum_{j=0}^{i-1} u(j)\right]\right) \tilde t^*(z) \\
    & = \tilde t(z) \sum_{i=0}^\infty \left(z^ir(i)\left[1 - \sum_{j=0}^{i-1} u(j)\right]\right),
\end{align*}
where the last step is possible because $\Tc^*$ is an identically distributed copy of $\Tc$. Replacing the second summand in the earlier expression and solving for $\tilde t(z)$ gives us the formula,
\begin{equation*}
    \tilde t(z) = \frac{\sum_{n=0}^\infty z^nu(n)\left[1 - \sum_{i=0}^n r(i)\right]}{1 - \sum_{i=0}^\infty \left(z^ir(i)\left[1 - \sum_{j=0}^{i-1} u(j)\right]\right)}.
\end{equation*}
This expression is actually sufficient for many useful calculations, but it can be convenient to rewrite it. From here, some simple algebra and another exchange of summation order gives us
\begin{align*}
    \tilde t(z) & = \frac{\sum_{n=0}^\infty z^nu(n) - \sum_{n=0}^\infty z^nu(n)\sum_{i=0}^n r(i)}{1 - \sum_{i=0}^\infty z^ir(i) + \sum_{i=0}^\infty z^ir(i)\sum_{j=0}^{i-1} u(j)} \\
    & = \frac{\tilde u(z) - \sum_{n=0}^\infty z^nu(n)\sum_{i=0}^n r(i)}{1 - \tilde r(z) + \sum_{j=0}^\infty u(j) \left[\tilde r(z) - \sum_{i=0}^j z^ir(i)\right]} \\
    & = \frac{\tilde u(z) - \sum_{n=0}^\infty z^nu(n)\sum_{i=0}^n r(i)}{1 - \tilde r(z)(1-\Ec_\Uc) - \sum_{n=0}^\infty u(n)\sum_{i=0}^n z^ir(i)}.
\end{align*}
\end{proof}

\subsection{Two common restart mechanisms}
With the formula for $\tilde t(z)$, it might behoove us to cover two common distributions for restart: a geometric distribution with constant rate, $\rho$, and a deterministic or sharp distribution at some constant time, $N$. These two distributions are of particular interest for several reasons, not least among them that we can actually compute some useful results.

\subsubsection{The geometric restart}
In this text, we define a geometric distribution by its cumulative mass function as $R(n) = 1-(1-\rho)^n$ for $n\in\naturals$ with constant rate parameter $\rho\in(0,1)$. We consider the limiting cases $\rho\to0$ and $\rho\to1$ to be no restart and restart every step, respectively. From the cumulative distribution, we can easily write down the probability mass function, $r(n) = \rho(1-\rho)^{n-1}$ for $n\in\integers^+\coloneqq\{1,2,3,\ldots\}$, and the $z$-transform, $\tilde r(z) = \frac{\rho z}{1-(1-\rho)z}$. An observation worth making for the geometric restart is that $\probability(\Rc\le\Uc)<1$ and restart is non-preemptive for $\rho\in(0,1)$. In the following sections, many results depend on non-preemptive restart, so the geometric distribution is often a good candidate for experimentation. We also want to draw attention to the fact that the support of $\Rc$ is $\integers^+$ in this definition. This is in contrast to some other works (e.g. \cite{bonomo2021first}) that use an alternative parameterization of the distribution so that $\Rc\in\naturals$.

\subsubsection{The sharp restart}
An even simpler distribution can be defined with cumulative mass function $R(n) = \mathbbm{1}_{[N,\infty)}(n) = \begin{cases}1 & n \ge N \\ 0 & n < N\end{cases}$ with $n\in\naturals$ for some parameter $N\in\integers^+$. This admits probability mass function $r(n) = \delta_{n,N}$ and PGF $\tilde r(z) = z^N$. In contrast to the geometric restart, this sharp restart can be preemptive. To illustrate:  consider the case where the underlying stochastic process can reach its target set no earlier than time $m$. If $N\le m$, then $\probability(\Rc\le\Uc)=1$, the FPUR cannot ever reach the target set, and $\Tc$ is almost surely infinite.

\section{Hitting Probabilities and Recurrence}\label{sec:hiting_probabilities}
With the generating function for $\Tc$ given by Lemma \ref{lemma-pdf_FPUR_gf}, we can look at how the hitting probability of the FP process is changed by adding an arbitrary restart mechanism. We simply evaluate the generating function at $z=1$, which recovers exactly $\Ec_\Tc = \sum_{n\in\naturals}t(n)$.

\begin{lemma}\label{lemma-hit_t}
The hitting probability of the FPUR process is given by 
\begin{equation*}
    \Ec_\Tc = \frac{\Ec_\Uc - \sum_{n=0}^\infty u(n)R(n)}{1 - \Ec_\Rc(1-\Ec_\Uc) - \sum_{n=0}^\infty u(n)R(n)},
\end{equation*}
when $d \coloneqq 1 - \Ec_\Rc(1-\Ec_\Uc) - \sum_{n=0}^\infty u(n)R(n)\neq0$. Otherwise $\Ec_\Tc=0$.
\end{lemma}
\begin{proof}
The formula itself is an immediate result of Lemma \ref{lemma-pdf_FPUR_gf}, so we merely address the circumstances under which $d$ equals $0$ and show that $\Ec_\Tc=0$ in that case. Setting the denominator to zero gives us
\begin{align*}
    0 & = 1 - \Ec_\Rc(1-\Ec_\Uc) - \sum_{n=0}^\infty u(n)R(n) \\
    1-\Ec_\Rc & = \sum_{n=0}^\infty u(n)R(n) - \Ec_\Uc\Ec_\Rc \\
    1-\Ec_\Rc & = \sum_{n=0}^\infty u(n)[R(n) - \Ec_\Rc].
\end{align*}
The left-hand side is clearly nonnegative, and the right-hand side is clearly nonpositive, which indicates that $d\ge0$, with equality only when both of the following conditions are met:
\begin{itemize}
    \item $\Ec_\Rc = 1$, and
    \item $R(n)=\Ec_\Rc$ for all $n$ in the support of $u(n)$.
\end{itemize}
Put plainly, this is the case in which the restart is assured and must occur with probability 1 before the underlying process has any chance to reach the target set, which is the definition of preemptive restart. Notice that $d>0$ whenever $\Ec_\Rc<1$. 
\end{proof}

This lemma indicates that, when restart is preemptive, the hitting probability of the FPUR process becomes 0, exactly as one might expect. In the case of non-preemptive restart, however, we can say a bit more.

\begin{theorem}\label{theo-T_iff_U_or_R}
    For a FPUR process with non-preemptive restart, $\Ec_\Tc=1$ iff at least one of $\Ec_\Uc$ and $\Ec_\Rc$ is $1$.
\end{theorem}

\begin{proof}
    First, suppose that $\Ec_\Tc=1$. Since the restart is non-preemptive, we may multiply both sides by the denominator to obtain the following sequence of equalities.
    \begin{align*}
        1 - \Ec_\Rc(1-\Ec_\Uc) - \sum_{n=0}^\infty u(n)R(n) & = \Ec_\Uc - \sum_{n=0}^\infty u(n)R(n) \\
        1 - \Ec_\Rc(1-\Ec_\Uc) & = \Ec_\Uc \\
        1 - \Ec_\Rc - \Ec_\Uc + \Ec_\Rc\Ec_\Uc & = 0 \\
        (1 - \Ec_\Rc)(1 - \Ec_\Uc) & = 0
    \end{align*}
    Thus we must have at least one of $\Ec_\Rc$ and $\Ec_\Uc$ equal to 1. \\
    Next, we have two cases:
    \begin{itemize}
        \item Suppose $\Ec_\Rc = 1$. \\ Then we have $\Ec_\Tc = \frac{\Ec_\Uc - \sum_{n=0}^\infty u(n)R(n)}{1 - 1\cdot(1-\Ec_\Uc) - \sum_{n=0}^\infty u(n)R(n)} = \frac{\Ec_\Uc - \sum_{n=0}^\infty u(n)R(n)}{\Ec_\Uc - \sum_{n=0}^\infty u(n)R(n)}=1$.
        \item Suppose $\Ec_\Uc = 1$. \\Then we have $\Ec_\Tc = \frac{1 - \sum_{n=0}^\infty u(n)R(n)}{1 - \Ec_\Rc(1-1) - \sum_{n=0}^\infty u(n)R(n)} = \frac{1 - \sum_{n=0}^\infty u(n)R(n)}{1 - \sum_{n=0}^\infty u(n)R(n)}=1$.
    \end{itemize}
\end{proof}

There are many consequences to this theorem, but we'll take a moment to note an important one.

\begin{proposition}\label{prop:rec}
    Given an underlying discrete stochastic process, $\mathcal{W}_n$, and a restart mechanism with $\Ec_\Rc=1$, any terminal point that can be reached in finite time by $\mathcal{W}_n$ becomes recurrent for $\mathcal{W}_n^*$, provided the restart is non-preemptive.
\end{proposition}

This proposition highlights the value of our geometric restart mechanism. Since the PGF for $\Rc$ is $\tilde r(z) = \frac{\rho z}{1-(1-\rho)z}$, we can immediately check that $\Ec_\Rc = \tilde r(1) = 1$. Since geometric restart is furthermore non-preemptive as discussed at the end of Section \ref{sec:PGF_for_T}, Proposition \ref{prop:rec} tells us that the FPUR is recurrent for every state it can reach in finite time, even when the underlying FP process is not!

\section{Hitting Times}\label{sec:hitting_times}
While understanding the hitting probability is a critical step in analyzing the first passage statistics of a stochastic process, our goal is often to compute the expected hitting time. One option is to differentiate the expression from Lemma \ref{lemma-pdf_FPUR_gf} and then evaluate at $z=1$, which gives expressions for $\expectation[\Tc]$ that are equivalent to those given in Section \ref{sec:definitions}. Depending on the complexity of $\tilde r(z)$, $\tilde u(z)$, $r(n)$ and $u(n)$, however, evaluating the derivative might be the easier approach. Any of these formulations will permit an extension to Proposition \ref{prop:rec}.

\begin{proposition}\label{prop:rec2}
    Given an underlying discrete stochastic process, $\mathcal{W}_n$, and a restart mechanism with $\expectation[\Rc]<\infty$, any terminal point that can be reached in finite time by $\mathcal{W}_n$ becomes positive recurrent for $\mathcal{W}_n^*$, provided the restart is non-preemptive.
\end{proposition}

\begin{proof}
    If we suppose that $\expectation[\Rc]<\infty$, then clearly $\Ec_\Rc = 1$ and we have recurrence by Proposition \ref{prop:rec}. To show positive recurrence, take equation (\ref{eq:min_formula)}): $\expectation[\Tc] = \frac{\expectation[\Uc \wedge \Rc]}{1-p_r}$. So long as $\expectation[\Rc]$ is finite and $p_r<1$, then $\expectation[\Tc]<\infty$.
\end{proof}

In general, it is not easy to compute expressions for the hitting time. Our two restarts from Section \ref{sec:PGF_for_T}, however, do allow for relatively simple formulations.

\subsection{Geometric restart} \label{sec-geom_formula}
Substituting $r(n) = \rho(1-\rho)^{n-1}$ for $n\ge1$ and $\tilde r(z) = \frac{\rho z}{1 - (1-\rho)z}$ into the formula from Lemma \ref{lemma-pdf_FPUR_gf} gives us the following PGF for $\Tc$:
\begin{equation*}
    \tilde t(z) = \frac{\tilde u((1-\rho)z)}{1-\frac{\rho z}{1-(1-\rho)z}\left(1  - \tilde u((1-\rho)z)\right)}.
\end{equation*}
Taking the derivative with respect to $z$ and then evaluating at $z=1$ gives us the formula for the mean hitting time below.
\begin{equation}\label{eq:hitting_time-geom}
    \expectation[\Tc] = \frac{1 - \tilde u(1-\rho)}{\rho \tilde u(1-\rho)}
\end{equation}
It's worth taking a moment to comment on how wonderful this expression is. It's not only concise, it also allows us to compute $\expectation[\Tc]$ with only the PGF for $\Uc$ \textemdash no need for taking further derivatives or computing any partial sums of the probability mass function.

We can also examine both of the limiting cases for $\rho$. Using L'H\^{o}pital's Rule to take $\rho\to0$, we find that $\expectation[\Tc]\to\expectation[\Uc]$. This matches nicely with our earlier interpretation: the $\rho\to0$ case should correspond to having no restart at all. In the other direction, taking $\rho\to1$ gives $\expectation[\Tc]\to\frac{1-u(0)}{u(0)}$, except that we assumed $u(0)=\probability(\Uc=0)=0$ as discussed in Section \ref{sec:introduction}. Thus, we have $\expectation[\Tc]\to\infty$ and restart is preemptive. We note that in \cite{bonomo2021first} Bonomo and Pal derive a similar result. It differs from ours slightly, but only as a result of the choice of strict inequality in (\ref{eq:recursive_definition}) and our subsequent parameterization of the geometric distribution starting at 1 instead of 0.

Before moving to the sharp restart, we recall that geometric restart has $\expectation[\Rc]=\frac{1}{\rho}<\infty$, and thus a FPUR equipped with geometrically distributed restart is in fact positive recurrent for any state the underlying process can reach in finite time by Proposition \ref{prop:rec2}. The sharp restart has the same property, but as previously noted can be preemptive, which we see in the next section.

\subsection{Sharp restart} \label{sec-sharp_formula}
Just as for the geometric restart, we insert our probability mass function and PGF into the formula from Lemma \ref{lemma-pdf_FPUR_gf}. With $r(n)=\delta_{n,N}$ and $\tilde r(z) = z^N$, we have 
\begin{align*}
    \tilde t(z) & = \frac{\tilde u(z) - \sum_{n=0}^\infty z^nu(n)\mathbbm{1}_{[N,\infty)}(n)}{1-z^N(1-\Ec_\Uc) - \sum_{n=0}^\infty u(n) z^N \mathbbm{1}_{[N,\infty)}(n)} \\
    & = \frac{\tilde u(z) - \sum_{n=N}^\infty z^nu(n)}{1-z^N(1-\Ec_\Uc) - z^N\sum_{n=N}^\infty u(n)} \\
    & = \frac{\sum_{n=0}^{N-1}z^nu(n)}{1 - z^N\left( 1 - \sum_{n=0}^{N-1}u(n) \right)}
\end{align*}
Differentiating and evaluating at $z=1$ gives us the mean first passage time (which we can recognize as (\ref{eq:min_formula)}) or (\ref{eq:num_restart_formula})  where $p_r = 1-U(N-1)$).
\begin{equation}\label{eq:hitting_time-sharp}
    \expectation[\Tc] = \frac{\sum_{n=0}^{N-1}nu(n) + N(1 - U(N-1))}{U(N-1)}
\end{equation}

Just as with the geometric restart, we're interested in the cases corresponding to instantaneous restart and no restart, in this case $N=1$ and $N\to\infty$, respectively. For the no restart case, we can see $\lim_{N\to\infty}\expectation[\Tc](N)=\expectation[\Uc]$, as expected. On the other hand $\expectation[\Tc]$ is undefined when $N=1$ since $U(0)$ is 0 by assumption. In that case, we have $\probability(\Rc\le\Uc)=1$, and restart is preemptive.

 Also, since $\expectation[\Rc] = N < \infty$, the sharp restart also satisfies Proposition \ref{prop:rec2} and guarantees positive recurrence when restart is non-preemptive, i.e. $N$ must be strictly larger than the smallest value in the support of $u(n)$, or equivalently $U(N-1)>0$.

\section{When is $\expectation[\Tc] < \expectation[\Uc]$?}\label{sec:ET<EU}
In the context of applications, we are often concerned with whether adding a restart mechanism will speed up or slow down the underlying process, i.e., reduce or increase the expected arrival time in the target set. We say that a restart is \underline{beneficial} if the FPUR equipped with this mechanism has a smaller mean first passage time than the underlying FP process. Our goal then is to characterize the circumstances under which a restart is beneficial.

One simple way of determining whether a restart is beneficial in some parameter range is to examine the sign of $\expectation[\Tc]-\expectation[\Uc]$. Clearly, when this difference is negative, the restart is beneficial. Unfortunately, this expression doesn't admit any concise or easily computable form to check for arbitrary underlying and restart processes. Below, however, we do consider a different kind of criterion for the geometric restart and a useful special case for the sharp.

\subsection{Geometric restart and the derivative condition}

Let $\Rc\sim\geom(\rho)$. Then taking limits as $\rho$ tends to 0 and 1 of $\expectation[\Tc]=\frac{1-\tilde u(1-\rho)}{\rho\tilde u(1-\rho)}$ yields
\begin{align*}
    \lim_{\rho\to0}\expectation[\Tc] & = \expectation[\Uc] \\
    \lim_{\rho\to1}\expectation[\Tc] & = \infty,
\end{align*}
which we already noted in Section \ref{sec-geom_formula}. Thus, we can say that $\expectation[\Tc]$ atarts at $\expectation[\Uc]$ with $\rho = 0$ and tends to infinity as $\rho\to1$, perhaps increasing non-monotonically. Differentiating $\expectation[\Tc](\rho)$ with respect to $\rho$ and taking the limit as $\rho\to0$ produces the expression $D\coloneqq\frac{2\tilde u'(1)^2-\tilde u''(1)}{2}$. If $D<0$, then there clearly exists some interval of $\rho$ values (specifically $(0,\hat\rho)$ for some $\hat\rho\in(0,1)$) with $\expectation[\Tc]<\expectation[\Uc]$. Note that $\expectation[\Tc]$ may not be convex in $\rho$, so the restart could also be beneficial on some of $(\hat\rho,1)$, and $D\ge0$ does not guarantee that restart is not beneficial for some interval. That is, $D<0$ is sufficient, but not necessary, to ensure an interval where $\expectation[\Tc]<\expectation[\Uc]$. In the event that we can demonstrate that $\expectation[\Tc](\rho)$ is convex, we know that $D<0$ iff there exists some $\hat\rho\in(0,1)$ such that restart is beneficial for $\rho\in(0,\hat\rho)$.\\
\begin{figure}[h]
\includegraphics[width=8.6cm]{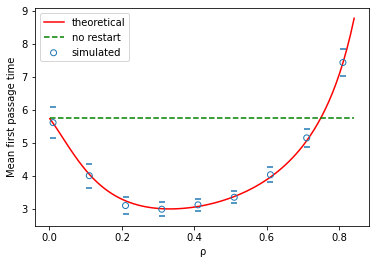}
\caption{\label{fig:counter_Dneg} This plot shows $\expectation[\Tc](\rho)$ for selected values of $\rho$ between .01 and .841. $D<0$ is sufficient to guarantee an interval where $\expectation[\Tc]<\expectation[\Uc]$. Simulated values were averaged over 2000 trials, and we include 99\% confidence intervals.}
\end{figure}\\
\textbf{Example with $D<0$ (Figure \ref{fig:counter_Dneg})}\\
Consider the case of an underlying process that finishes at either time $1$ with probability $\frac{3}{4}$ or $20$ with probability $\frac{1}{4}$. The PGF for $\Uc$ is $\tilde u(z) = \frac{1}{4}\left(3z+z^{20}\right)$, so that $\expectation[\Uc]=5\frac{3}{4}$. We then compute $D=-\frac{231}{16}<0$, so there must be some interval starting at 0 for which $\expectation[\Tc]<\expectation[\Uc].$

\textbf{Example with $D>0$ (Figure \ref{fig:counter_Dpos})}\\
Now consider the preceding case with the probabilities reversed. The PGF for $\Uc$ is $\tilde u(z) = \frac{1}{4}\left(z+3z^{20}\right)$, and clearly $\expectation[\Uc]=15\frac{1}{4}$. We then compute $D=\frac{2\left(\frac{61}{4}\right)^2 - \frac{1140}{4}}{2}=\frac{1441}{16}>0$. Since $\expectation[\Tc](\rho)$ is not convex in $\rho$, despite $D>0$, we see in Figure \ref{fig:counter_Dpos} that there is a region of beneficial restart (albeit not starting at 0).\\

\begin{figure}[h]
\includegraphics[width=8.6cm]{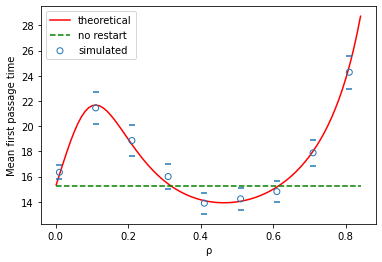}
\caption{\label{fig:counter_Dpos} This plot shows $\expectation[\Tc](\rho)$ for selected values of $\rho$ between .01 and .841. Despite $D>0$, we can clearly see an interval where $\expectation[\Tc]<\expectation[\Uc]$. Simulated values were averaged over 2000 trials, and we include 99\% confidence intervals.}
\end{figure}

\subsection{Sharp restart and piecewise linear behavior}
We have no analogous trick for the sharp reset, but there is a special case of the underlying process worth mentioning. If we reconsider (\ref{eq:hitting_time-sharp}), we see that $N$ appears as limits of sums of $u(n)$ and multiplied by $\frac{1-U(N-1)}{U(N-1)}$, which is the expected number of restarts. 
\begin{equation*}
    \expectation[\Tc](N) = \frac{1 - U(N-1)}{U(N-1)}N + \frac{\sum_{n=0}^{N-1}nu(n)}{U(N-1)}
\end{equation*}
Importantly, if there are gaps in the support of $u(n)$, the cumulative mass function and the index-weighted sum are both constant across that gap. This implies that $\expectation[\Tc](N)$ is linear in $N$ across gaps in the support of the probability mass function, and it has a strictly positive slope, which decreases monotonically as $N\to\infty$. This may seem to be very specific, but it's not uncommon for gaps to exist in the support of $u(n)$. Even the simple example of the FP time to 0 a nearest-neighbor random walk on the integer lattice has this property. If the walk begins at 1, then $\Uc$ is almost surely odd. The implication above indicates that $\expectation[\Tc](N)$ would always be larger for $N=2k$ than for $N=2k+1$, $k\in\naturals.$

\section{The ``Cycle Trap"}\label{sec:cycle_trap}
Now we introduce an example that serves quite well to explore the concepts of the previous sections. The process we name the ``cycle trap" has a finite state space labeled by the integers from $-L$ to $M$, where $L,M\in\integers^+$. The process begins at the vertex labeled 0, and terminates at vertex $-L$, with movement between the vertices largely deterministic. At 0, the process can move to vertex -1 with probability $p$ or to 1 with probability $q \coloneqq 1-p$. If it moves to -1, then it continues moving to $-L$ one vertex at a time with probability 1 at each step. If it instead moves to 1, then it continues similarly to vertex $M$ before cycling back 0, as seen in Figure \ref{fig:cycle_trap_diagram}.

\begin{figure}[h]
\includegraphics[width=8.6cm]{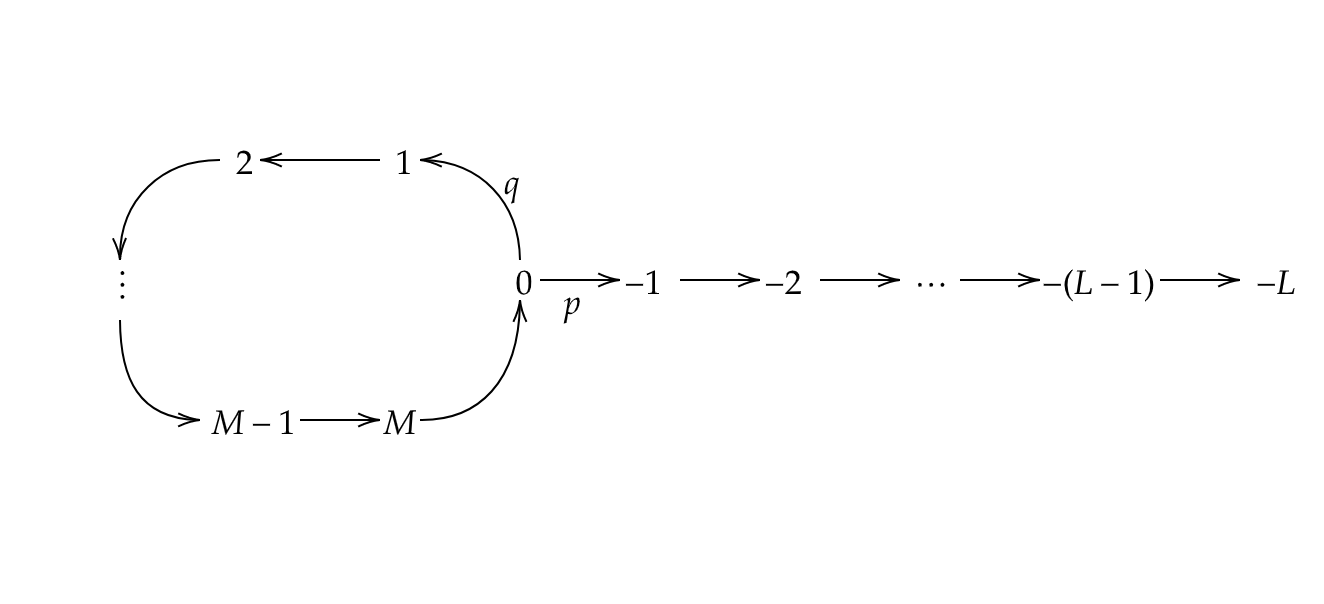}
\caption{\label{fig:cycle_trap_diagram} The Cycle Trap}
\end{figure}

The above formulation admits a PGF for the first passage time to $-L$ of $\tilde u(z) = \frac{pz^L}{1-qz^{M+1}}$. Provided that $q\neq1$, evaluating $\tilde u$ and its derivative at $z=1$ yields a hitting probability of $\Ec_\Uc = 1$ and a mean first passage time of $\expectation[\Uc] = L + \frac{q}{p}(M+1)$. After deriving a few formulas, adjusting these parameters ($p$, $L$ and $M$) will allow us to explore some nuances of beneficial restart. Just as in Section \ref{sec:ET<EU}, we shall focus on the geometric and sharp distributions.

\subsubsection{Geometric Restart}
Returning to equation (\ref{eq:hitting_time-geom}), we can plug in $\tilde u(z) = \frac{pz^L}{1 - qz^{M+1}}$ to find 
\begin{equation*}
    \expectation[\Tc] = \frac{1 - p(1-\rho)^L - q(1-\rho)^{M+1}}{p\rho(1-\rho)^L},
\end{equation*}
with $\rho$ as our rate parameter for the restart. We first note that L'H\^{o}pital's Rule confirms $\lim_{\rho\to0}\expectation[\Tc]=L+\frac{q}{p}(M+1)=\expectation[\Uc]$. Since $u(0)=0$ and $\lim_{\rho\to1}\expectation[\Tc]=\infty$, we can use the derivative criterion to check for beneficial restart. We also show that $\expectation[\Tc]$ is convex in $\rho$ by rewriting
\small
\begin{align*}
    \expectation[\Tc] & = \frac{p(1-(1-\rho)^L) + q(1-(1-\rho)^{M+1})}{p \rho (1-\rho)^L} \\
    & = (1-\rho)^{-L} + (1-\rho)^{-L+1} + \ldots + (1-\rho)^{-1} \\
    & \quad + \frac{q}{p}\left((1-\rho)^{-L} + (1-\rho)^{-L+1} + \ldots + (1-\rho)^{M-L}\right).
\end{align*}
\normalsize
Here we see that $\expectation[\Tc]$ is a positive linear combination of integer powers of $(1-\rho)$ and is thus convex in $\rho$. This observation indicates that the derivative criterion is biconditional: the restart will have some beneficial interval of $\rho$ values if and only if $D<0$.

\begin{figure}[h]
\includegraphics[width=8.6cm]{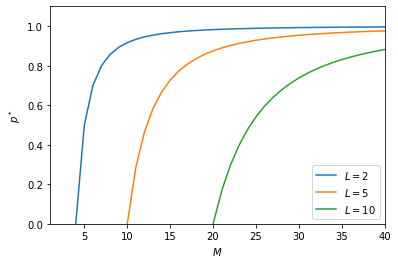}
\caption{\label{fig:cycletrap_p_bound}For selected values of $L$, the upper bounds for the bias as a function of $M$ are given.}
\end{figure}

We compute that the condition $D<0$ is equivalent to $(L+1)L+\frac{q}{p}(M+1)(2L-M)<0$, which is true exactly when $p<p^*\coloneqq\frac{(M+1)(M-2L)}{(M+1)(M-2L)+L(L+1)}$, as shown in Figure \ref{fig:cycletrap_p_bound}, where we recall that $p$ is the probability that the process moves to exit the trap. This expression is clearly equal to 0 when $M=2L$ and increases monotonically towards 1 as $M/L$ gets large. When $\frac{M}{L}>2$, this upper bound is strictly greater than 0, and there exists a range of values for $p$ such that there can be beneficial restart. Examples of both the $D<0$ and $D>0$ cases can be seen in Figure \ref{fig:cycletrap}.

\begin{figure}[h]
\includegraphics[width=8.6cm]{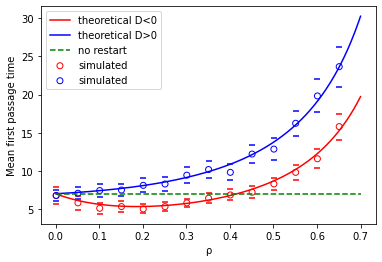}
\caption{\label{fig:cycletrap} Two examples of the cycletrap with geometric restart. The $D<0$ case has parameter values $(p,L,M)=(\frac{3}{4},2,14)$, and the $D>0$ case has parameter values selected to produce the same mean hitting time for the underlying process, with $(p,L,M) = (\frac{1}{2},2,4)$. Both blow up as $\rho\to1$. The simulated values were averaged over 500 trials, and we include 99\% confidence intervals.}
\end{figure}
Speaking broadly, this condition means that $M$ must be at least twice $L$ for beneficial restart to be possible, but as bias toward the exit increases (i.e. for larger values of $p$), greater ratios of $M$ to $L$ are required. On the other hand, for any positive value of $p$, the ``cost" of falling into the trap denoted by $\frac{M}{L}$ can be increased sufficiently to allow for beneficial restart. 

The sharp restart is less restrictive than the geometric, requiring only that $M>L$ with no dependence on $p$, as we see in the next section.

\subsubsection{Sharp Restart}
In order to calculate some of the terms in equation (\ref{eq:hitting_time-sharp}), we must first obtain the probability mass function by extracting the coefficients of $\tilde u(z)$. Then we can compute the cumulative mass function, $U(n)$, by summing these coefficients. Thankfully, in the case of the cycle trap, this is straightforward, as we can write $\tilde u(z) = pz^L\sum_{k=0}^\infty \left(qz^{M+1}\right)^k$, which gives $u(n)=pq^{\frac{n-L}{M+1}}$ for $n \equiv L \Mod{M+1}$ and 0 otherwise. Computing the partial and index-weighted partial sums results in 
\begin{equation*}
    \expectation[\Tc] = L + \frac{q^{K+1}N + \frac{q}{p}(M+1)\left(Kq^{K+1}-(K+1)q^K + 1\right)}{1-q^{K+1}},
\end{equation*}
where $K = \left\lfloor \frac{N-1-L}{M+1} \right\rfloor$. We immediately notice when looking at the plots of $\expectation[\Tc](N)$ in Figures \ref{fig:cyc_sharp_1} and \ref{fig:cyc_sharp_2} the piecewise linear behavior described in Section \ref{sec:ET<EU}, with $\expectation[\Tc]$ increasing linearly over gaps in the support of $u(n)$.

We also give a criterion that determines the relationship between $\expectation[\Tc]$ and $\expectation[\Uc]$:
\begin{itemize}
    \item When $L \ge M$, we have $\expectation[\Tc] \ge \expectation[\Uc]$ for all values of $N$. In particular, $L>M$ implies $\expectation[\Tc] > \expectation[\Uc]$.
    \item When $L < M$, beginning at $N=L+1$, every $M+1$ values of $N$ will have first $M-L$ values with $\expectation[\Tc] < \expectation[\Uc]$, then one value of $N$ with $\expectation[\Tc] = \expectation[\Uc]$ followed by $L$ values of $N$ with $\expectation[\Tc] > \expectation[\Uc]$. 
\end{itemize}
This can be verified directly by subtracting $\expectation[\Uc]$ from $\expectation[\Tc]$.
\scriptsize
\begin{align*}
    \expectation[\Tc] - \expectation[\Uc] & = L + \frac{q^{K+1}N + \frac{q}{p}(M+1)\left(Kq^{K+1}-(K+1)q^K + 1\right)}{1-q^{K+1}} \\
    & \quad - \left(L + \frac{p}{q}(M+1)\right) \\
    & = \frac{q^{K+1}}{1-q^{K+1}}\left(N-(M+1)(K+1)\right) \\
    & = \frac{q^{K+1}(M+1)}{1-q^{K+1}}\left(\frac{N-1-M}{M+1}-\left\lfloor{\frac{N-1-L}{M+1}}\right\rfloor\right)
\end{align*}
\normalsize
The last line establishes the previous dichotomy. When $L \ge M$, we have $\expectation[\Tc]\ge\expectation[\Uc]$ with equality only when $L=M$ and $N\equiv 0 \Mod{M+1}$. Beneficial restart occurs only when $L < M$, for values of $N$ that satisfying $a(M+1) + (L+1) \le N < (a+1)(M+1)$ for $a\in\naturals$. It's worth mentioning, further, that there are no cases where $\expectation[\Tc]\le\expectation[\Uc]$ for all values of $N$. It is always true the $\expectation[\Tc](N)>\expectation[\Uc]$ for $N \equiv L \Mod{M+1}$.\\
\begin{figure}[h]
\includegraphics[width=8.6cm]{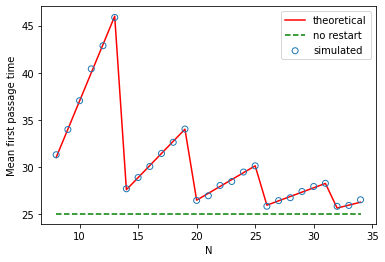}
\caption{\label{fig:cyc_sharp_1} The cycle trap with $(p,L,M) = (.25,7,5)$:  Since $L> M$, we observe the behavior in which the hitting time decreases towards $\expectation[\Uc]$, never going below. The simulated values were averaged over 50000 trials.}
\end{figure}\\
\textbf{Example with no beneficial restart (Figure \ref{fig:cyc_sharp_1})}\\
To demonstrate the first kind of behavior, we pick parameter values $(p,L,M)=(.25,7,5)$, so that $L\ge M$. Looking towards Figure \ref{fig:cyc_sharp_1}, we see a case where $\expectation[\Tc]>\expectation[\Uc]$ for all $N$.\\
\textbf{Example with beneficial restart (Figure \ref{fig:cyc_sharp_2})}\\
To demonstrate the second kind, we pick parameter values $(p,L,M)=(.25,5,10)$, so that $L<M$. The restart is clearly beneficial for certain values of $N$ and not for others.
\begin{figure}[h]
\includegraphics[width=8.6cm]{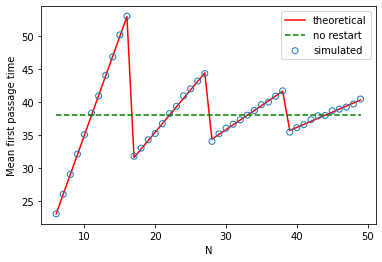}
\caption{\label{fig:cyc_sharp_2} The cycle trap with $(p,L,M) = (.25,5,10)$:  Since $L< M$, we observe the behavior in which the hitting time moves back and forth across $\expectation[\Uc]$, with the difference decaying to 0. The simulated values were averaged over 50000 trials.}
\end{figure}\\
An interesting observation one might make after seeing these figures is that $\expectation[\Tc]$ increases across gaps in the support of $u(n)$, as we demonstrated in Section \ref{sec:hitting_times}, but appears to decrease at each value of $N$ that increases $K$, which are precisely those values of $N$ such that $u(N)>0$. To see that this is true, we can pick, for some $a\in\naturals$,
\begin{align*}
    N_1 = L+a(M+1), & \quad K_1 = a-1 \\
    N_2 = L+1+a(M+1), & \quad K_2 = a.
\end{align*}
Notice that $N_1$ has been picked such that $u(N_1)>0$ and we may compute $\expectation[\Tc](N_1)-\expectation[\Tc](N_2)$. 

After some algebra, we see that this difference is $\frac{\left(1+q^{a+1}\right)q^aL}{\left(1-q^a\right)\left(1-q^{a+1}\right)}+\frac{q^{a+1}}{1-q^{a+1}}M$, which is positive, meaning that the graph of $\expectation[\Tc](N)$ will always demonstrate this saw-tooth behavior, decreasing after a value in the support and increasing otherwise. This is, however, not true for all underlying processes, as we will show in the next example.

\section{Biased random walk on $\naturals$}
We now turn our attention to a more interesting and classical example. One of the early major investigations into stochastic restart was \cite{evans2011diffusion}, in which Evans and Majumdar considered symmetric diffusion on the half-line with 0 as an absorbing boundary. This is a well-known problem in the study of first passage processes, and has an infinite mean hitting time. Evans and Majumdar show that introducing a restart mechanism can make this mean hitting time finite. More recently, Christophorov showed in \cite{christophorov2020peculiarities} that the discrete analogue of Evan's and Majumdar's model has some interesting differences in behavior. In this section, we hope to build on the preceding works to discuss the case of the discrete biased random walk on $\naturals$.
\subsection{The PGF for the underlying process}
Just as we outlined in Section \ref{sec:PGF_for_T}, we want to first obtain the generating function for the FP process, in this case, the biased random walk on $\naturals$ with 0 as an its terminal state. The derivation of this equation is too long to include here, though the result has been known at least since \cite{feller1957introduction}. For a starting point $m\in\integers^+$, with $p\in(0,1)$ the probability of moving towards 0 and $q\coloneqq1-p$ its complement, we have the PGF for the first passage time to 0 below.

\begin{equation*}
    \tilde u(z) = \left(\frac{1-\sqrt{1-4pqz^2}}{2qz}\right)^m
\end{equation*}

Using this expression, we can immediately check the hitting probability and mean hitting time of the FP process.

\begin{align*}
    \Ec_\Uc & = \min\left(1,\frac{p}{q}\right)^m \\
    \expectation[\Uc] & = 
    \begin{cases}
        \frac{m}{p-q} & p>q \\
        \infty & p \le q
    \end{cases}
\end{align*}

Clearly, the case with $p=q$ has a hitting probability of 1, meaning that it will almost surely reach 0 in finite time, but a mean hitting time of infinity, even starting a single vertex away. This is the case studied with restart by Evans and Majumdar (in the continuous time paradigm), and by Christophorov (in the discrete).

By Sections \ref{sec:hiting_probabilities} and \ref{sec:hitting_times}, we know that the $q > p$ case is also dramatically changed by adding a restart mechanism. Provided that the restart is non-preemptive and that $\Ec_\Rc = 1$, we know that the hitting probability of the resultant FPUR is 1, that is, $\Ec_\Tc = 1$. Furthermore, so long as $\expectation[\Rc]<\infty$ (such as with the geometric restart) the mean hitting time also becomes finite. Thus, even for the ``misbehaved" case where the underlying process is biased away, adding a suitable restart mechanism means that 0 becomes positive recurrent. It then remains, just as before, to determine the conditions under which restart may be beneficial.

\subsection{When is $\expectation[\Tc] < \expectation[\Uc]?$}
As discussed in Section \ref{sec:ET<EU}, we often frame our questions about FPUR around whether the restart can reduce the mean hitting time. For $q\ge p$, equivalently $p \le \frac{1}{2}$, we have established that $\expectation[\Tc] < \expectation[\Uc]$ for any non-preemptive restart by the simple fact that $\expectation[\Tc]$ is finite and $\expectation[\Uc]$ is not. When $p > q$, however, things get more interesting.
\subsubsection{Geometric Restart}
Recall from Section \ref{sec:hitting_times} that for a geometrically distributed restart time with parameter $\rho\in(0,1)$, we can express the mean hitting time by
$\expectation[\Tc] = \frac{1 - \tilde u(1-\rho)}{\rho \tilde u(1-\rho)}$. This formula, combined with our expression for $\tilde u(z)$ allows us to offer an explicit expression for the mean hitting time of the FPUR.

\begin{equation*}
    \expectation[\Tc] = \frac{\left( 2q(1-\rho)\right)^m - \left( 1 - \sqrt{1 - 4pq(1-\rho)^2} \right)^m }{\rho \left( 1 - \sqrt{1 - 4pq(1-\rho)^2} \right)^m}
\end{equation*}
Just as with the cycle trap, we can utilize the derivative criterion, but demonstrating convexity is not so simple. However, we can instead argue for a biconditional result another way. Defining $\xi = 2(1-\rho)$, we can write, for $p>q$,
\small
\begin{align*}
    F(\xi) & = \rho(\expectation[\Tc]-\expectation[\Uc]) = \left(\frac{q\xi}{1-\sqrt{1-pq\xi^2}}\right)^m - \frac{(2-\xi)m}{2(p-q)} - 1 \\
    & = \left(\frac{1+\sqrt{1-pq\xi^2}}{p\xi}\right)^m - \frac{(2-\xi)m}{2(p-q)} - 1.
\end{align*}
\normalsize
As $0 \le \xi \le 2$, we examine the boundary cases, finding $F(2) = \left(\frac{1-\sqrt{1-4pq}}{2p}\right) - 1 = 0$, and $\lim_{\xi\to0}F(\xi)=+\infty$. Computing the first and second derivative of $F$ with respect to $\xi$, we find $F'(2)=0$ and that $\sgn(F''(\xi))$ can be reduced to $\sgn(\psi(\xi))$, where $\psi(\xi)\coloneqq m\sqrt{1-pq\xi^2}+1-2pq\xi^2$. Since $\psi(\xi)$ is decreasing with $\psi(0) = m+1>0$, we have only two possibilites:
\begin{itemize}
    \item $\psi(2)\ge0$, in which case $F''(\xi)>0$ for all $\xi\in(0,2)$, implying $F'(\xi)<0$ and $F(\eta)>0$. In this case, $\expectation[\Tc]>\expectation[\Uc]$ for all $\rho$ and restart is never beneficial.
    \item $\psi(2)<0$, in which case there exists some inflection point $\xi_0\in(0,2)$ such that $F''(\xi)>0$ for all $\xi<\xi_0$ and $F''(\xi)<0$ for all $\xi>\xi_0$. Consequently, $F(\xi)$ attains a negative minimum at some $\xi_1<\xi_0$, and there exists a $\xi^*\in(0,\xi_1)$ such that $F(\xi)>0$ for $\xi\in(0,\xi^*)$ and $F(\xi)<0$ for $\xi\in(\xi^*,2)$. This is precisely the case where there exists some $\rho^*$ such that restart is beneficial for $\rho\in(0,\rho^*)$.
\end{itemize}
We can examine the inequality $\psi(2)<0$ to find that there is a region of beneficial restart when $m\sqrt{1-4p}+1-8pq<0$, which is equivalent to $m<m^*\coloneqq\frac{8p(1-p)-1}{2p-1}$. In particular, if $p\ge\frac{3}{4}$, then $m^*\le1$ and geometric restart is not beneficial for any value of $\rho$. Some algebra allows us to rewrite this criterion as $p<p^*\coloneqq\frac{4-m+\sqrt{m^2+8}}{8}$, which is equivalent to $D<0$ from the derivative test. Note that this expression is decreasing in $m$ and asymptotically approaches $\frac{1}{2}$ as we can see in Figure \ref{fig:brwn_p_bound}. In other words, if the process begins further from 0, the bias must be less for restart to be beneficial. Framed differently, for even a very small bias towards 0, starting sufficiently far means that restart cannot be beneficial.
\begin{figure}[h]
\includegraphics[width=8.6cm]{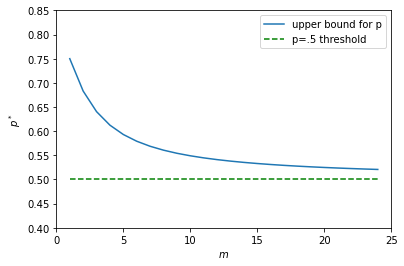}
\caption{\label{fig:brwn_p_bound} For a given starting point $m$, only values of $p$ below this bound, $p^*$ will admit a range of $\rho$ values for which restart is beneficial.}
\end{figure}\\

\subsubsection{Sharp Restart}
Unfortunately, the sharp restart case is not so easily tractable here as it was for the cycle trap, though we can make a couple simple observations. First, restart is preemptive for $N \le m$. Second, $u(n)$ will only be nonzero for values of $n$ with the same parity as $m$. Thus, $\expectation[\Tc]$ will be increasing at least every other value of $N$.

From equation (\ref{eq:hitting_time-sharp}), we need to compute $U(n) = \sum_{i=0}^nu(i)$ and $U_w(n) = \sum_{i=0}^niu(i)$, the forms of which can be seen below.
\begin{align*}
    U(n) & = mp^m\sum_{k=0}^{\lfloor\frac{n-m}{2}\rfloor} \frac{(pq)^k}{m+2k}{m+2k \choose k} \\
    U_w(n) & = mp^m\sum_{k=0}^{\lfloor\frac{n-m}{2}\rfloor} (pq)^k{m+2k \choose k}
\end{align*}
We have not found a way to write these sums explicitly, but we can compute them numerically to plot the mean FP time of the FPUR as a function of $N$. In so doing, we observe some apparently distinct behaviors as shown in Figures \ref{fig:brwn_sharp_1}, \ref{fig:brwn_sharp_2} and \ref{fig:brwn_sharp_3}.

\begin{figure}[h]
\includegraphics[width=8.6cm]{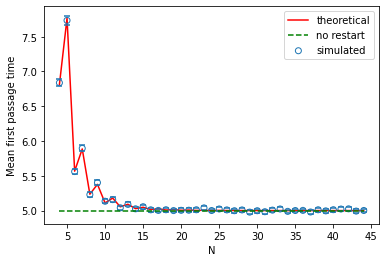}
\caption{\label{fig:brwn_sharp_1}For parameter values $(p,m)=(.8,3)$, we see that $\expectation[\Tc]$ starts high and then decreases towards $\expectation[\Uc]$. In this case, numerical solutions suggest that $\expectation[\Tc]>\expectation[\Uc]$ for all $N$, but we weren't able to verify that. This is clearly reminiscent of the cycle trap with $L>M$.}
\end{figure}
\begin{figure}[h]
\includegraphics[width=8.6cm]{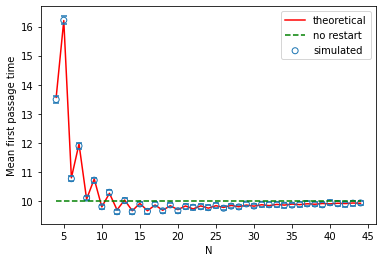}
\caption{\label{fig:brwn_sharp_2}By decreasing $p$ to $.65$, we see that $\expectation[\Tc]$ still begins above $\expectation[\Uc]$, but passes below before increasing towards the limiting value. Changing the value of $m$ seems only to change the steepness of this initial drop. The sawtooth behavior appears to continue indefinitely, but this has not been confirmed.}
\end{figure}
\begin{figure}[h]
\includegraphics[width=8.6cm]{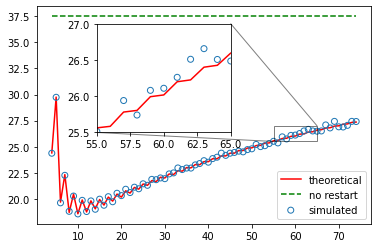}
\caption{\label{fig:brwn_sharp_3}As $p$ gets closer to .5, in this case $(p,m)=(.54,3)$, we can observe an interesting deviation from the behavior of the cycle trap: as $N$ increases, the sawtooth behavior eventually disappears, and $\expectation[\Tc]$ increases monotonically. Whether this behavior also exists in the previous example for some larger value of $N$ is unknown, but it is a distinct departure from what we saw with the cycle trap.}
\end{figure}

\section{Conclusion}
In this paper, we have presented some analysis of stochastic restart on FP processes in the case of discrete time. In particular, we have examined two restart distributions that permit some explicit formulas:  the sharp and the geometric. We note that whenever this restart is non-preemptive and occurs almost surely in finite time, then the mean FP time of the process with restart will also be finite, even if that isn't true for the underlying process. If we further assume that the restart distribution has finite mean, then the mean FP of the process with restart does as well. This provides us with a method of forcing any process to be positive recurrent (with the possible expense of extending its mean hitting time):  just add a non-preemptive restart mechanism with finite mean.

We then explored some of these mechanisms in the context of two examples, introducing the cycle trap stochastic process and examining the well-studied process of the biased random walk on the half-line. In particular, we showed that the addition of a restart mechanism can have a profound and surprising effect on the mean hitting time of a stochastic process.

\nocite{*}

\bibliography{references}

\begin{thebibliography}{12}%
\makeatletter
\providecommand \@ifxundefined [1]{%
 \@ifx{#1\undefined}
}%
\providecommand \@ifnum [1]{%
 \ifnum #1\expandafter \@firstoftwo
 \else \expandafter \@secondoftwo
 \fi
}%
\providecommand \@ifx [1]{%
 \ifx #1\expandafter \@firstoftwo
 \else \expandafter \@secondoftwo
 \fi
}%
\providecommand \natexlab [1]{#1}%
\providecommand \enquote  [1]{``#1''}%
\providecommand \bibnamefont  [1]{#1}%
\providecommand \bibfnamefont [1]{#1}%
\providecommand \citenamefont [1]{#1}%
\providecommand \href@noop [0]{\@secondoftwo}%
\providecommand \href [0]{\begingroup \@sanitize@url \@href}%
\providecommand \@href[1]{\@@startlink{#1}\@@href}%
\providecommand \@@href[1]{\endgroup#1\@@endlink}%
\providecommand \@sanitize@url [0]{\catcode `\\12\catcode `\$12\catcode
  `\&12\catcode `\#12\catcode `\^12\catcode `\_12\catcode `\%12\relax}%
\providecommand \@@startlink[1]{}%
\providecommand \@@endlink[0]{}%
\providecommand \url  [0]{\begingroup\@sanitize@url \@url }%
\providecommand \@url [1]{\endgroup\@href {#1}{\urlprefix }}%
\providecommand \urlprefix  [0]{URL }%
\providecommand \Eprint [0]{\href }%
\providecommand \doibase [0]{https://doi.org/}%
\providecommand \selectlanguage [0]{\@gobble}%
\providecommand \bibinfo  [0]{\@secondoftwo}%
\providecommand \bibfield  [0]{\@secondoftwo}%
\providecommand \translation [1]{[#1]}%
\providecommand \BibitemOpen [0]{}%
\providecommand \bibitemStop [0]{}%
\providecommand \bibitemNoStop [0]{.\EOS\space}%
\providecommand \EOS [0]{\spacefactor3000\relax}%
\providecommand \BibitemShut  [1]{\csname bibitem#1\endcsname}%
\let\auto@bib@innerbib\@empty
\bibitem [{\citenamefont {Gupta}\ \emph {et~al.}(2014)\citenamefont {Gupta},
  \citenamefont {Majumdar},\ and\ \citenamefont
  {Schehr}}]{gupta2014fluctuating}%
  \BibitemOpen
  \bibfield  {author} {\bibinfo {author} {\bibfnamefont {S.}~\bibnamefont
  {Gupta}}, \bibinfo {author} {\bibfnamefont {S.~N.}\ \bibnamefont
  {Majumdar}},\ and\ \bibinfo {author} {\bibfnamefont {G.}~\bibnamefont
  {Schehr}},\ }\bibfield  {title} {\bibinfo {title} {Fluctuating interfaces
  subject to stochastic resetting},\ }\href@noop {} {\bibfield  {journal}
  {\bibinfo  {journal} {Physical review letters}\ }\textbf {\bibinfo {volume}
  {112}},\ \bibinfo {pages} {220601} (\bibinfo {year} {2014})}\BibitemShut
  {NoStop}%
\bibitem [{\citenamefont {Reuveni}\ \emph {et~al.}(2014)\citenamefont
  {Reuveni}, \citenamefont {Urbakh},\ and\ \citenamefont
  {Klafter}}]{reuveni2014role}%
  \BibitemOpen
  \bibfield  {author} {\bibinfo {author} {\bibfnamefont {S.}~\bibnamefont
  {Reuveni}}, \bibinfo {author} {\bibfnamefont {M.}~\bibnamefont {Urbakh}},\
  and\ \bibinfo {author} {\bibfnamefont {J.}~\bibnamefont {Klafter}},\
  }\bibfield  {title} {\bibinfo {title} {Role of substrate unbinding in
  michaelis--menten enzymatic reactions},\ }\href@noop {} {\bibfield  {journal}
  {\bibinfo  {journal} {Proceedings of the National Academy of Sciences}\
  }\textbf {\bibinfo {volume} {111}},\ \bibinfo {pages} {4391} (\bibinfo {year}
  {2014})}\BibitemShut {NoStop}%
\bibitem [{\citenamefont {Rold{\'a}n}\ \emph {et~al.}(2016)\citenamefont
  {Rold{\'a}n}, \citenamefont {Lisica}, \citenamefont {S{\'a}nchez-Taltavull},\
  and\ \citenamefont {Grill}}]{roldan2016stochastic}%
  \BibitemOpen
  \bibfield  {author} {\bibinfo {author} {\bibfnamefont {{\'E}.}~\bibnamefont
  {Rold{\'a}n}}, \bibinfo {author} {\bibfnamefont {A.}~\bibnamefont {Lisica}},
  \bibinfo {author} {\bibfnamefont {D.}~\bibnamefont {S{\'a}nchez-Taltavull}},\
  and\ \bibinfo {author} {\bibfnamefont {S.~W.}\ \bibnamefont {Grill}},\
  }\bibfield  {title} {\bibinfo {title} {Stochastic resetting in backtrack
  recovery by rna polymerases},\ }\href@noop {} {\bibfield  {journal} {\bibinfo
   {journal} {Physical Review E}\ }\textbf {\bibinfo {volume} {93}},\ \bibinfo
  {pages} {062411} (\bibinfo {year} {2016})}\BibitemShut {NoStop}%
\bibitem [{\citenamefont {Huang}(2007)}]{huang2007effect}%
  \BibitemOpen
  \bibfield  {author} {\bibinfo {author} {\bibfnamefont {J.}~\bibnamefont
  {Huang}},\ }\bibfield  {title} {\bibinfo {title} {The effect of restarts on
  the efficiency of clause learning.},\ }in\ \href@noop {} {\emph {\bibinfo
  {booktitle} {IJCAI}}},\ Vol.~\bibinfo {volume} {7}\ (\bibinfo {year} {2007})\
  pp.\ \bibinfo {pages} {2318--2323}\BibitemShut {NoStop}%
\bibitem [{\citenamefont {Evans}\ \emph {et~al.}(2020)\citenamefont {Evans},
  \citenamefont {Majumdar},\ and\ \citenamefont
  {Schehr}}]{evans2020stochastic}%
  \BibitemOpen
  \bibfield  {author} {\bibinfo {author} {\bibfnamefont {M.~R.}\ \bibnamefont
  {Evans}}, \bibinfo {author} {\bibfnamefont {S.~N.}\ \bibnamefont
  {Majumdar}},\ and\ \bibinfo {author} {\bibfnamefont {G.}~\bibnamefont
  {Schehr}},\ }\bibfield  {title} {\bibinfo {title} {Stochastic resetting and
  applications},\ }\href@noop {} {\bibfield  {journal} {\bibinfo  {journal}
  {Journal of Physics A: Mathematical and Theoretical}\ }\textbf {\bibinfo
  {volume} {53}},\ \bibinfo {pages} {193001} (\bibinfo {year}
  {2020})}\BibitemShut {NoStop}%
\bibitem [{\citenamefont {Pal}\ and\ \citenamefont
  {Reuveni}(2017)}]{pal2017first}%
  \BibitemOpen
  \bibfield  {author} {\bibinfo {author} {\bibfnamefont {A.}~\bibnamefont
  {Pal}}\ and\ \bibinfo {author} {\bibfnamefont {S.}~\bibnamefont {Reuveni}},\
  }\bibfield  {title} {\bibinfo {title} {First passage under restart},\
  }\href@noop {} {\bibfield  {journal} {\bibinfo  {journal} {Physical review
  letters}\ }\textbf {\bibinfo {volume} {118}},\ \bibinfo {pages} {030603}
  (\bibinfo {year} {2017})}\BibitemShut {NoStop}%
\bibitem [{\citenamefont {Bonomo}\ and\ \citenamefont
  {Pal}(2021{\natexlab{a}})}]{bonomo2021first}%
  \BibitemOpen
  \bibfield  {author} {\bibinfo {author} {\bibfnamefont {O.~L.}\ \bibnamefont
  {Bonomo}}\ and\ \bibinfo {author} {\bibfnamefont {A.}~\bibnamefont {Pal}},\
  }\bibfield  {title} {\bibinfo {title} {First passage under restart for
  discrete space and time: Application to one-dimensional confined lattice
  random walks},\ }\href@noop {} {\bibfield  {journal} {\bibinfo  {journal}
  {Physical Review E}\ }\textbf {\bibinfo {volume} {103}},\ \bibinfo {pages}
  {052129} (\bibinfo {year} {2021}{\natexlab{a}})}\BibitemShut {NoStop}%
\bibitem [{\citenamefont {Evans}\ and\ \citenamefont
  {Majumdar}(2011)}]{evans2011diffusion}%
  \BibitemOpen
  \bibfield  {author} {\bibinfo {author} {\bibfnamefont {M.~R.}\ \bibnamefont
  {Evans}}\ and\ \bibinfo {author} {\bibfnamefont {S.~N.}\ \bibnamefont
  {Majumdar}},\ }\bibfield  {title} {\bibinfo {title} {Diffusion with
  stochastic resetting},\ }\href@noop {} {\bibfield  {journal} {\bibinfo
  {journal} {Physical review letters}\ }\textbf {\bibinfo {volume} {106}},\
  \bibinfo {pages} {160601} (\bibinfo {year} {2011})}\BibitemShut {NoStop}%
\bibitem [{\citenamefont
  {Christophorov}(2020)}]{christophorov2020peculiarities}%
  \BibitemOpen
  \bibfield  {author} {\bibinfo {author} {\bibfnamefont {L.}~\bibnamefont
  {Christophorov}},\ }\bibfield  {title} {\bibinfo {title} {Peculiarities of
  random walks with resetting in a one-dimensional chain},\ }\href@noop {}
  {\bibfield  {journal} {\bibinfo  {journal} {Journal of Physics A:
  Mathematical and Theoretical}\ }\textbf {\bibinfo {volume} {54}},\ \bibinfo
  {pages} {015001} (\bibinfo {year} {2020})}\BibitemShut {NoStop}%
\bibitem [{\citenamefont {Feller}(1957)}]{feller1957introduction}%
  \BibitemOpen
  \bibfield  {author} {\bibinfo {author} {\bibfnamefont {W.}~\bibnamefont
  {Feller}},\ }\href@noop {} {\emph {\bibinfo {title} {An introduction to
  probability theory and its applications}}}\ (\bibinfo  {publisher} {Wiley},\
  \bibinfo {year} {1957})\BibitemShut {NoStop}%
\bibitem [{\citenamefont {Redner}(2001)}]{redner2001guide}%
  \BibitemOpen
  \bibfield  {author} {\bibinfo {author} {\bibfnamefont {S.}~\bibnamefont
  {Redner}},\ }\href@noop {} {\emph {\bibinfo {title} {A guide to first-passage
  processes}}}\ (\bibinfo  {publisher} {Cambridge university press},\ \bibinfo
  {year} {2001})\BibitemShut {NoStop}%
\bibitem [{\citenamefont {Bonomo}\ and\ \citenamefont
  {Pal}(2021{\natexlab{b}})}]{bonomo2021p}%
  \BibitemOpen
  \bibfield  {author} {\bibinfo {author} {\bibfnamefont {O.~L.}\ \bibnamefont
  {Bonomo}}\ and\ \bibinfo {author} {\bibfnamefont {A.}~\bibnamefont {Pal}},\
  }\bibfield  {title} {\bibinfo {title} {The p\'olya and sisyphus lattice
  random walks with resetting--a first passage under restart approach},\
  }\href@noop {} {\bibfield  {journal} {\bibinfo  {journal} {arXiv preprint
  arXiv:2106.14036}\ } (\bibinfo {year} {2021}{\natexlab{b}})}\BibitemShut
  {NoStop}%
\end{thebibliography}%

\end{document}